\documentclass[12pt]{amsart}

\usepackage{color,graphicx,amssymb,latexsym,amsfonts,txfonts,amsmath,amsthm}
\usepackage{pdfsync}
\usepackage{enumitem}
\usepackage{amsmath,amscd}
\usepackage[all,cmtip]{xy}

\input epsf
\newcommand{\s}{\vspace{0.3cm}}
\newtheorem{theo}{Theorem}
\newtheorem{coro}[theo]{Corollary}

\newtheorem{lemm}[theo]{Lemma}

\newtheorem{rema}[theo]{Remark}

\begin{document}

\title[Decomposing the Jacobian of Fermat curves ]{A remark on the decomposition of the Jacobian variety of Fermat curves of prime degree}
\author{Ruben A. Hidalgo and Rub\'{\i} E. Rodr\'{\i}guez}


\address{Departamento de Matem\'atica y Estad\'{\i}stica, Universidad de La Frontera. Casilla 54-D, 4780000 Temuco, Chile}
\email{ruben.hidalgo@ufrontera.cl, rubi.rodriguez@ufrontera.cl}
\thanks{Partially supported by Fondecyt grants 1150003 and 1141099}

\begin{abstract}
Recently, Barraza-Rojas have described the action of the full automorphisms group on the Fermat curve of degree $p$, for $p$ a prime integer, and obtained the group algebra decomposition of the corresponding Jacobian variety. In this short note we observe that the factors in such a decomposition are given by the Jacobian varieties of certain $p$-gonal curves.
\end{abstract}

\maketitle

\section{Introduction}
Let $S$ be a compact Riemann surface and let $G$ be a group of conformal automorphisms of $S$. The induced action of $G$ on $H_{1}(S,{\mathbb Z})$ (the first homology group of $S$) provides the \textit{rational representation}, and its action on $H^{1,0}(S)$ (the space of holomorphic one-forms on $S$) provides the \textit{analytic} representation; they in turn  induce an action of $G$ on the Jacobian variety $JS$ of $S.$

A  relationship between the rational irreducible representations of $G$ and the $G$-invariant factors in the isotypical decomposition of $JS$ was given in  \cite{C-R,L-R}; providing a $G$-equivariant  isogenous decomposition of $JS$ as
$$JS \sim J(S/G) \times B_{2}^{u_2} \times \cdots \times B_{r}^{u_r},$$
where each $B_{i}$ is a certain abelian subvariety of $JS$. This decomposition was studied in \cite{B-R} when $S$ is the classical Fermat curve $F_p:=x^{p}+y^{p}+z^{p}=0$, where $p \geq 5$ is a prime integer; this is a Riemann surface of genus $(p-1)(p-2)/2$. In there,
by considering $G={\rm Aut}(F_{p})$, the full group of conformal automorphisms of $F_{p}$,
it was obtained that

\begin{equation} \label{dec1}
JF_{p} \sim
\left\{\begin{array}{ll}
B^{6} \times B_{0}^{3} \times B_{1}^{6} \times \cdots \times B_{(p-7)/6}^{6}, & p \equiv 1 \mod(3)\\
B_{0}^{3} \times B_{1}^{6} \times \cdots \times B_{(p-5)/6}^{6}, & p \equiv 2 \mod(3)
\end{array}
\right.
\end{equation}
where each $B_{i}$ has dimension $(p-1)/2$ and $B$ has dimension $(p-1)/6$.

Another decomposition of $JF_{p}$, using techniques of number theory, was provided in \cite{A}. In such a decomposition, there are $(p-2)$ factors, each one of dimension $(p-1)/2$, and each of them turns out to be a subvariety of CM-type. It is also described there which of these subvarieties are simple.

In this paper we observe that the factors in the decomposition of $JF_{p}$ given in \eqref{dec1}  are Jacobian varieties of certain cyclic $p$-gonal curves, as shown in Theorem \ref{teo2}.
For instance, $B_{0}=JC_{1}$, where
$C_{1}: y^{p}=x(x-1)$ (a hyperelliptic Riemann surface) and, for $p \equiv 1 \mod(3)$, $B=JE_{\gamma_{p}}$, where $E_{\gamma_{p}}=C_{\gamma_{p}}/\langle R \rangle$,
$\gamma_{p}^{2}+\gamma_{p}+1 \equiv 0 \mod (p)$,    $C_{\gamma_{p}}: y^{p}=x^{\gamma_{p}}(x-1)$ and
$R(x,y)=(1/(1-x), (-1)^{\epsilon} x^{(\gamma_{p}^{2}+\gamma_{p}+1)/p}/y^{\gamma_{p}+1})$, with $\epsilon=1$ if $\gamma_{p}$ is even and $\epsilon=2$ otherwise.

\section{Certain cyclic $p$-gonal curves}
In this section, $p \geq 5$ will denote a prime integer and $\mathbb{F}_{p}$ will denote the finite field of cardinality $p$.

\s
\subsection{An action of the symmetric group $\mathfrak{S}_{3}$}
In the set $X_{p}=\{1,\ldots,p-2\} \subset \mathbb{F}_{p}$ we consider the action of the symmetric group on three letters
$$\mathfrak{S}_{3}=\langle U,V: U^{3}=V^{2}=(UV)^{2}=1\rangle$$
given by
$$U(\alpha)=-(1+\alpha)^{-1}; V(\alpha)=\alpha^{-1},$$
where, for each element $\beta \in X_{p}$ we take the inverse $\beta^{-1}$ and the opposite $-\beta$ in the field $\mathbb{F}_{p}$.

For each $\alpha \in X_{p}$ we consider its $\mathfrak{S}_{3}$-orbit given by
$${\mathcal O}(\alpha)=\{\alpha, \alpha^{-1}, -(1+\alpha), -(1+\alpha)^{-1}, -\alpha^{-1} (1+\alpha), -\alpha (1+\alpha)^{-1}\}.$$

One may observe that the only orbit of cardinality three is given by ${\mathcal O}(1)=\{1,p-2,(p-1)/2\}$. If $p \equiv 2 \mod(3)$ and $\alpha \notin {\mathcal O}(1)$, then the cardinality of ${\mathcal O}(\alpha)$ is six. If $p \equiv 1 \mod(3)$, then there is one orbit of cardinality two, given by the orbit of $\gamma_{p} \in X_{p}$ satisfying the quadratic equation $\gamma_{p}^{2}+\gamma_{p}+1=0$ (in $\mathbb{F}_{p}$); in this case, ${\mathcal O}(\gamma)=\{\gamma_{p}, \gamma_{p}^{-1}=p-1-\gamma_{p}\}$. If $\alpha \notin \{1,p-2,(p-1)/2,\gamma_{p}, \gamma_{p}^{-1}\}$, then its orbit has cardinality six.

\s
\subsection{Some cyclic $p$-gonal curves}
For each $\alpha \in \{1,2,\ldots,p-2\}$, we consider the following cyclic $p$-gonal curve
$$C_{\alpha}: y^{p}=x^{\alpha}(x-1).
$$

The curve $C_{\alpha}$ admits as conformal automorphism the order $p$ transformation $T(x,y)=(x, \omega_{p} y)$, where $\omega_{p}=e^{2 \pi i/p}$. The map $\pi:C_{\alpha} \to \widehat{\mathbb C}$ given by $\pi(x,y)=x$ is a regular branched covering, with deck group the cyclic group generated by $T$, whose branch values are $\infty$, $0$ and $1$, each one with branch order $p$. It then follows from the Riemann-Hurwitz formula that $C_{\alpha}$ has genus $(p-1)/2$.

These curves were studied by Lefschetz in \cite{L}, where he showed that for $p \geq 11$ the group generated by $T$ is the only subgroup of order $p$ in the group of automorphisms of $C_{\alpha}$.

\noindent
\begin{lemm}
Two curves $C_{\alpha_{1}}$ and $C_{\alpha_{2}}$ are isomorphic Riemann surfaces if and only if ${\mathcal O}(\alpha_{1})={\mathcal O}(\alpha_{2})$.
\end{lemm}
\begin{proof}
First, for $\alpha, \beta \in \{1,\ldots,p-1\}$ so that $\alpha + \beta \nequiv 0 \mod(p)$, we consider the algebraic curve
$$C_{\alpha,\beta}: y^{p}=x^{\alpha} (x-1)^{\beta}.$$

Let us consider the triangular Fuchsian group $\Delta_{p}=\langle x_{0}, x_{1}, x_{\infty}: x_{0}^{p}=x_{1}^{p}=x_{\infty}^{p}=x_{0} x_{1} x_{\infty}=1\rangle$.
The order $p$ cyclic branched cover $(x,y) \in C_{\alpha,\beta} \to x \in \widehat{\mathbb C}$ corresponds to the kernel of the monodromy
$$\Theta_{\alpha,\beta}:\Delta_{p} \to {\mathbb F}_{p}$$
$$x_{0} \mapsto \alpha^{-1}, \quad x_{1} \mapsto \beta^{-1}.$$

As the kernel of $\Theta_{\alpha,\beta}$ does not change if we post-compose with an automorphisms of ${\mathbb F}_{p}$ (which is given by ${\mathbb F}_{p}^{*}$), we may see that
if $\delta \in \mathbb{F}_{p}-\{0\}=\{1,\ldots,p-1\}$, then $C_{\delta\alpha, \delta\beta}$ and $C_{\alpha,\beta}$ are isomorphic. In particular, by taking $\delta=\beta^{-1}$, we have that $C_{\alpha,\beta}$ and $C_{\alpha\beta^{-1},1}$ are isomorphic.

As a consequence of results in \cite{Gabino} or in \cite{R-R}, the curves $C_{\alpha_{1}}$ and $C_{\alpha_{2}}$ are isomorphic if and only if there is an isomorphism $\phi:C_{\alpha_{1}} \to C_{\alpha_{2}}$ and a M\"obius transformation $\Phi:\widehat{\mathbb C} \to \widehat{\mathbb C}$ so that $\pi \circ \phi = \Phi \circ \pi$, where $\pi(x,y)=x$. In particular, $\Phi$ must keep invariant the set
$\{\infty,0,1\}$; so $\Phi \in  \mathfrak{S}_{3} \subset \rm{Aut}(\widehat{\mathbb C})$.

\begin{enumerate}
\item If $\Phi(x)=x$, then $\alpha_{2}=\alpha_{1}$.

\item If $\Phi(x)=1/x$, then $\phi(C_{\alpha_{1}})=C_{-(1+\alpha_{1})}$.

\item If $\Phi(x)=1-x$, then $\phi(C_{\alpha_{1}})=C_{1,\alpha_{1}} \cong C_{\alpha^{-1}}$.

\item If $\Phi(x)=x/(x-1)$, then $\phi(C_{\alpha_{1}})=C_{\alpha_{1},-(1+\alpha_{1})} \cong C_{-\alpha_{1}(1+\alpha_{1})^{-1}}$.

\item If $\Phi(x)=1/(1-x)$, then $\phi(C_{\alpha_{1}})=C_{1,-(1+\alpha_{1})} \cong C_{-(1+\alpha_{1})^{-1}}$.

\item If $\Phi(x)=(x-1)/x$, then $\phi(C_{\alpha_{1}})=C_{-(1+\alpha_{1}),\alpha_{1}} \cong C_{-\alpha_{1}^{-1}(1+\alpha_{1})}$.
\end{enumerate}
\end{proof}

\begin{rema}
\begin{enumerate}
\item The curve $C_{1}$ is hyperelliptic and the hyperelliptic involution is given by \\
$J(x,y)=(1-x,y)$. The Riemann surface $C_1$ can be also described by the hyperelliptic curve
$$w^{2}=u^{p}-1.$$

\item If $p \equiv 1 \mod(3)$, then the curve $C_{\gamma_{p}}$ admits a conformal automorphism of order $3$ given by
$$R(x,y)=(1/(1-x), (-1)^{\epsilon} x^{(\gamma_{p}^{2}+\gamma_{p}+1)/p}/y^{\gamma_{p}+1}),$$
where $\epsilon=1$ if $\gamma_{p}$ is even and $\epsilon=2$ otherwise.
We note that $R$ acts with exactly two fixed points and that the quotient $C_{\gamma_{p}}/\langle R \rangle$ has genus $(p-1)/6$.
\end{enumerate}
\end{rema}

\s

\section{Jacobian variety and Kani-Rosen decomposition theorem}
\subsection{Principally polarized abelian varieties}
A {\it principally polarized abelian variety} of dimension $g \geq 1$ is a pair $A=(T,Q)$, where $T={\mathbb C}^{g}/L$ is a complex torus of dimension $g$ and $Q$ (called a {\it principal polarization} of $A$) is a positive-definite Hermitian product in ${\mathbb C}^{g}$ whose imaginary part ${\rm Im}(Q)$ has integral values over elements of the lattice $L$, and such that there is a basis of $L$ for which ${\rm Im}(Q)$ is given by the matrix
$$\left( \begin{array}{cc}
0_{g} & I_{g}\\
-I_{g} & 0_{g}
\end{array}
\right) \ ,
$$
where $I_g$ denotes the $g \times g$ identity matrix and $0_g$ denotes the $g \times g$  zero matrix.

Two complex tori $T_{1}$ and $T_{2}$ are called {\it isogenous} if there is a
non-constant surjective morphism $h : T_{1} \to T_{2}$ with finite kernel; in this case $h$
is called an {\it isogeny} (or an \textit{isomorphism}  if the kernel is trivial). We will write $T_1 \sim T_2$.

A principally polarized abelian variety $A$ is called {\it decomposable} if it is isogenous to the product of complex tori of smaller dimensions (otherwise, it is said to be {\it simple}). It is called {\it completely decomposable} if it is isogenous to the product of elliptic curves (abelian varieties of dimension $1$).

A general result in this respect is the following.

\noindent
\begin{theo}[Poincar\'e's complete reducibility theorem \cite{Poincare}]
If $A$ is a principally polarized abelian variety, then there exist simple polarized abelian varieties $A_{1}, \ldots, A_{s}$ and positive integers $n_{1},\ldots,n_{s}$ such that $A$ is isogenous to the product $A_{1}^{n_{1}} \times \cdots A_{s}^{n_{s}}$. Moreover, the $A_{j}$ and $n_{j}$ are unique up to isogeny and permutation of the factors.
\end{theo}

\s

\subsection{The Jacobian variety}
Let $S$ be a closed Riemann surface of genus $g \geq 1$. Its first homology group $H_{1}(S,{\mathbb Z})$ is isomorphic, as a ${\mathbb Z}$-module, to ${\mathbb Z}^{2g}$ and the complex vector space $H^{1,0}(S)$ of its holomorphic $1$-forms is isomorphic to ${\mathbb C}^{g}$. There is a natural injective map
\begin{align*}
\iota : H_{1}(S,{\mathbb Z}) & \hookrightarrow \left( H^{1,0}(S) \right)^{*} \quad \mbox{(the dual vector space of $H^{1,0}(S)$)} \\
\alpha & \mapsto \int_{\alpha}
\end{align*}

The image $\iota(H_{1}(S,{\mathbb Z}))$ is a lattice in $\left( H^{1,0}(S) \right)^{*}$ and the quotient $g$-dimensional torus
$$JS=\left( H^{1,0}(S) \right)^{*}/\iota(H_{1}(S,{\mathbb Z}))$$
is called the \textit{Jacobian variety} of $S$.  The geometric intersection product on $H_{1}(S,{\mathbb Z})$ induces a principal polarization on $JS$.

If we fix a point $p_{0} \in S$, then there is a natural holomorphic embedding $$\rho_{p_{0}}:S \to JS$$
defined by $\rho_{p_{0}}(p)=\int_{\alpha}$, where $\alpha$ is an arc in $S$ connecting $p_{0}$ to $p$.

If we choose a symplectic homology basis for $S$, say $\{\alpha_{1},\ldots,\alpha_{g},\beta_{1},\ldots,\beta_{g}\}$ (that is, a basis for $H_{1}(S,{\mathbb Z})$ such that the intersection products $\alpha_{i} \cdot \alpha_{j}=\beta_{i} \cdot \beta_{j}=0$ and $\alpha_{i} \cdot \beta_{j}=\delta_{ij}$, where $\delta_{ij}$ is the Kronecker delta function), we may find a dual basis $\{\omega_{1},\ldots,\omega_{g}\}$ (i.e. a basis of $H^{1,0}(S)$ such that $\int_{\alpha_{i}} \omega_{j}=\delta_{ij}$).

If we now consider the Riemann period matrix $\Omega=(I_g \; Z)_{g \times 2g}$, where  the Riemann matrix $Z$ is given by
$$Z=\left( \int_{\beta_{j}} \omega_{i} \right)_{g \times g} \in {\mathfrak H}_{g},$$
then the $2g$ columns of $\Omega$ generate a lattice in ${\mathbb C}^{g}$. The quotient torus ${\mathbb C}^{g}/\Omega$ is isomorphic to $JS$; both have the same
principal polarization.

\subsection{Kani-Rosen's decomposition theorem}
As a consequence of Poincar\'e complete reducibility theorem, the Jacobian variety of a closed Riemann surface can be decomposed, up to isogeny, into a product of simple sub-varieties. To obtain such a decomposition is not an easy job, but there are general results which permit to work in that direction (to obtain a decomposition into smaller dimension sub-varieties).

The following result, due to Kani and Rosen \cite{K-R}, provides sufficient conditions for the Jacobian variety of a closed Riemann surface to decompose into the product of the Jacobian varieties of suitable quotient Riemann surfaces. If $K<{\rm Aut}(S)$, we denote by $g_{K}$ the genus of the quotient orbifold $S/K$ and by $S_{K}$ the underlying Riemann surface structure of the orbifold $S/K$.

\noindent
\begin{theo}[Kani-Rosen's decomposition theorem \cite{K-R}]
Let $S$ be a closed Riemann surface of genus $g \geq 1$ and let $H_{1},\ldots,H_{r}\leq {\rm Aut}(S)$ such that:
\begin{enumerate}
\item $H_{i} H_{j}=H_{j} H_{i}$, for all $i,j =1,\ldots,r$;
\item there are non-zero integers  $n_{1},\ldots, n_{r}$ satisfying
\begin{enumerate}
\item $\sum_{i,j=1}^{r} n_{i}n_{j} g_{H_{i}H_{j}}=0$, and
\item for every $i=1,\ldots,r$, it also holds that $\sum_{j=1}^{r} n_{j} g_{H_{i}H_{j}}=0$.
\end{enumerate}
\end{enumerate}

Then
$$\prod_{n_{i}>0} \left( JS_{H_{i}} \right)^{n_{i}} \sim  \prod_{n_{j}<0} \left( JS_{H_{j}} \right)^{-n_{j}}.$$
\end{theo}

\s

If in the above theorem we set $H_{r}=\{1\}$, $n_{1}=\cdots=n_{r-1}=-1$ and $n_{r}=1$, then we have the following consequence, that we will use for our  Fermat curves.

\s
\noindent
\begin{coro}[\cite{K-R}]\label{coroKR}
Let $S$ be a closed Riemann surface of genus $g \geq 1$ and let $H_{1},\ldots,H_{s}<{\rm Aut}(S)$ be such that:
\begin{enumerate}
\item $H_{i} H_{j}=H_{j} H_{i}$, for all $i,j =1,\ldots,s$;

\medskip

\item $g_{H_{i}H_{j}}=0$, for $1 \leq i < j \leq s$;

\item $\displaystyle g=\sum_{j=1}^{s} g_{H_{j}}$.
\end{enumerate}

Then
$$JS \sim  \prod_{j=1}^{s} JS_{H_{j}}.$$
\end{coro}

\s

\section{A decomposition of the Jacobian variety of Fermat curve of prime degree}
Let $p \geq 5$ be a prime integer and let $F_{p}: \{ x^{p}+y^{p}+z^{p}=0\} \subset {\mathbb P}_{\mathbb C}^{2}$ be the classical Fermat curve of degree $p$. It is a well known fact that ${\rm Aut}(F_{p}) \cong {\mathbb Z}_{p}^{2} \rtimes D_{3}$, where
$D_{3}$ is the dihedral group of order six generated by $u([x:y:z])=[z:x:y]$ and $v([x:y:z])=[y:x:z]$, and  ${\mathbb Z}_{p}^{2}$ is generated by $a_{1}([x:y:z])=[\omega_{p} x:y:z]$ and $a_{2}([x:y:z])=[x:\omega_{p} y:z]$, where $\omega_{p}=e^{2 \pi i/p}$. Define $a_3$ by $a_{1}a_{2}a_{3}=1$.

Inside ${\rm Aut}(F_{p})$ we consider the abelian group $H=\langle a_{1},a_{2} \rangle \cong {\mathbb Z}_{p}^{2}$. The only elements in $H$ different from the identity acting with fixed points on $F_p$ are $a_{1}^{k}, a_{2}^{k}, a_{3}^{k}$, where $k=1,\ldots,p-1$. The map
$$\pi:F_{p} \to \widehat{\mathbb C}: [x:y:z] \mapsto -(y/x)^{p}$$
is a regular branched cover with deck group $H$ and branch set  $\{ \infty , 0 ,1 \}$. In fact, the image under $\pi$ of the set of fixed points of $a_1$ is $\infty$, the image of the fixed points of $a_2$ is $0$ and the image of the fixed points of $a_1 a_2$ is $1$.

The branched covering $\pi$ is a highest abelian branched covering of the Riemann sphere branched at $\infty$, $0$ and $1$ with orders equal to $p$ (this comes from the fact that $F_{p}$ is uniformized by the derived subgroup of the triangular group $\langle z_1, z_2, z_3: z_{1}^{p}=z_{2}^{p}=z_{3}^{p}=z_{1}z_{2}z_{3}=1\rangle$). As we have an order $p$ cyclic branched covering $\pi_{\alpha}:C_{\alpha} \to \widehat{\mathbb C}$, with branch values exactly the above ones, there must be a subgroup $K_{\alpha}$ of $H$, acting freely, of order $p$ (so a cyclic group) so that $F_{p}/K_{\alpha}$ corresponds to $C_{\alpha}$ and the cyclic cover $\pi_{\alpha}$ is defined by the cyclic group $H/K_{\alpha}$. Now, in $H$ there are exactly $(p-2)$ subgroups isomorphic to ${\mathbb Z}_{p}$ acting freely on $F_{p}$; these being
$$H_{j}=\langle a_{1}a_{2}^{1+j}\rangle, \quad j=1,\ldots,p-2.$$

The Riemann surface $R_{j}=F_{p}/H_{j}$ has genus $(p-1)/2$ and it admits a cyclic group of order $p$, generated by an automorphism $\tau$ induced by $a_{2}$ and quotient being the Riemann sphere, with corresponding cover branched over $\infty$, $0$ and $1$.
The condition $a_{1}a_{2}^{1+j} \in H_{j}$ asserts that $a_{1}$ induces the automorphism $\tau^{-(1+j)}$. This information ensures that $R_{j}$ corresponds to the curve $y^{p}=x^{-(1+j)}(x-1)$. In this way, for each $\alpha \in \{1,\ldots,p-2\}$, $C_{\alpha}=F_{p}/H_{p-1-\alpha}$. Note that $C_{\gamma_{p}}=F_{p}/H_{\gamma_{p}^{-1}}$ (since $\gamma_{p}^{-1}=p-1-\gamma_{p}$ in $\mathbb{F}_{p}$).

Now, since (i) $H_{i}H_{j}=H_{j}H_{i}$, (ii) $F_{p}/H_{i}H_{j}$ has genus zero for $i \neq j$ and (iii) the sum of the genera of all $(p-2)$ surfaces $F_{p}/H_{j}$ equal the genus of $F_{p}$, Corollary \ref{coroKR} implies the following result.

\s
\noindent
\begin{theo}\label{teo1}
Let $p \geq 5$ be a prime integer.
\begin{enumerate}
\item If $p \equiv 1 \mod(3)$, let $\alpha_{1},\ldots,\alpha_{(p-7)/6}$ be a maximal collection in $X_{p}$ of points with different orbits of cardinality six, and let $\gamma_{p} \in X_{p}$ denote a solution of $\gamma_{p}^{2}+\gamma_{p}+1=0$ in $\mathbb{F}_{p}$. Then
$$JF_{p} \sim JC_{1}^{3} \times JC_{\gamma_{p}}^{2} \times JC_{\alpha_{1}}^{6} \times \cdots \times JC_{\alpha_{(p-7)/6}}^{6}.$$

\item If $p \equiv 2 \mod(3)$, let $\alpha_{1},\ldots,\alpha_{(p-5)/6}$ be a maximal collection in $X_{p}$ of points with different orbits of cardinality six. Then
$$JF_{p} \sim JC_{1}^{3} \times JC_{\alpha_{1}}^{6} \times \cdots \times JC_{\alpha_{(p-5)/6}}^{6}.$$

\end{enumerate}
\end{theo}

\s

\subsection{The Jacobian variety of $C_{\gamma_{p}}$}
In the case $p \equiv 1 \mod(3)$, we have two solutions $\gamma_{p}$ in $\{1,\ldots,p-2\}$ to the quadratic equation $\gamma_{p}^{2}+\gamma_{p}+1 \equiv 0 \mod(p)$ (these are inverses of each other in $\mathbb{F}_{p}$).

Besides the order $p$ automorphism $T(x,y)=(x, \omega_{p} y)$, the
curve $C_{\gamma_{p}}$ also admits the order $3$ automorphism
$$R(x,y)=(1/(1-x), (-1)^{\epsilon} x^{(\gamma_{p}^{2}+\gamma_{p}+1)/p}/y^{\gamma_{p}+1}),$$
(where $\epsilon=1$ if $\gamma_{p}$ is even and $\epsilon=2$ otherwise). The automorphism $R$ acts with exactly two fixed points and the quotient $C_{\gamma}/\langle R \rangle$ has genus $(p-1)/6$.

It can be seen that $R \circ T=T^{\gamma_{p}^{2}} \circ R$. In particular, $\langle T \rangle$ is a normal subgroup of $\langle T,R \rangle$ and we have that $\langle T,R\rangle=\langle T \rangle \rtimes \langle R \rangle \cong {\mathbb Z}_{p} \rtimes {\mathbb Z}_{3}$. We also have the equality  $T^{-l} \circ R \circ T^{l}=T^{l(\gamma_{p}^{2}-1)} \circ R$.

Let us consider the subgroups of $\textup{Aut}(C_{\gamma_{p}})$ given as follows.
$$K_{1}=\langle R \rangle=\{I, R, R^{2}\}, $$
$$K_{2}=\langle T^{\gamma_{p}^{2}-1} \circ R \rangle=\{I,T^{\gamma_{p}^{2}-1} \circ R, T^{\gamma_{p}^{4}-1} \circ R^{2} \} = T^{-1} \, K_1 \, T,$$
$$K_{3}=\langle T^{2(\gamma_{p}^{2}-1)} \circ R \rangle=\{I,T^{2(\gamma_{p}^{2}-1)} \circ R, T^{2(\gamma_{p}^{4}-1)} \circ R^{2} \} = T^{-2} \, K_1 \, T^2 .$$

Thus we  see that $K_{i}K_{j}=K_{j}K_{i}$ (this follows from the fact that $\gamma_{p}^{3} \equiv 1 \mod(p)$). As $C_{\gamma_{p}}/K_{i}$ has genus $(p-1)/6$ and $g_{C_{\gamma_{p}}/K_iK_j}=0$  for $i \neq j$, we may apply Corollary \ref{coroKR} to conclude that
$$JC_{\gamma_{p}} \sim JE_{\gamma_{p}}^{3},$$
where $E_{\gamma_{p}}=C_{\gamma_{p}}/K_{1}$. In this way, Theorem \ref{teo1} can be written as follows.

\s
\noindent
\begin{theo}\label{teo2}
Let $p \geq 5$ be a prime integer.
\begin{enumerate}
\item If $p \equiv 1 \mod(3)$, let $\alpha_{1},\ldots,\alpha_{(p-7)/6}$ be a maximal collection in $X_{p}$ of points with different orbits and each one of cardinality $6$ and let $\gamma_{p} \in X_{p}$ solution of $\gamma_{p}^{2}+\gamma_{p}+1=0$ in $\mathbb{F}_{p}$. Then
$$JF_{p} \sim JC_{1}^{3} \times JE_{\gamma_{p}}^{6} \times JC_{\alpha_{1}}^{6} \times \cdots \times JC_{\alpha_{(p-7)/6}}^{6}.$$

\item If $p \equiv 2 \mod(3)$, let
    $\alpha_{1},\ldots,\alpha_{(p-5)/6}$ be a maximal
    collection in $X_{p}$ of points with different orbits
    and each one of cardinality $6$. Then
$$JF_{p} \sim JC_{1}^{3} \times JC_{\alpha_{1}}^{6} \times \cdots \times JC_{\alpha_{(p-5)/6}}^{6}.$$

\end{enumerate}
\end{theo}

\s

\section{Comparing the three isogenous decompositions of $JF_p$}

As we mentioned in the introduction, two decompositions of the Jacobians of the Fermat curves $F_p$ were known: those given in \cite{A} and in \cite{B-R}; in the previous section we have given a new decomposition.

 We now prove that all three decompositions coincide for the case $p$ prime.

\begin{theo}
a) The decomposition for $JF_p$ given in Theorem \ref{teo2} is
the group algebra decomposition of \eqref{dec1}.
In other words, the factors in the decomposition of $JF_p$
given by Theorem \ref{teo2} are isogenous to the factors in the
decomposition \eqref{dec1}, respectively as follows.

\begin{align*}
  B_0 & \sim JC_1 \\
  B & \sim JE_{\gamma_p} \\
  B_j & \sim JC_{\alpha_j} \; , \  \text{ for } 1 \leq  j \leq N,
\end{align*}
where
$$
N = \left\{
      \begin{array}{ll}
        (p-7)/6, & \hbox{if $p \equiv 1 \mod(3)$;} \\
        (p-5)/6, & \hbox{if $p \equiv 2 \mod(3)$.}
      \end{array}
    \right.
$$

b) The decomposition for $JF_p$ given in Theorem \ref{teo1} is
the decomposition given in \cite{A}.
\end{theo}

\begin{proof}
a) The decomposition of $JF_p$ given in \eqref{dec1} is the
group algebra decomposition of $JF_p$, where each factor
corresponds to a unique rational irreducible representation of
Aut($F_p$)  \cite{B-R}.

An algorithm to find explicit factors in the general group algebra
decomposition for a principally polarized abelian variety with
group action was given in \cite{C-R}. In this case, we can
apply the algorithm by noticing that the factors in the
decomposition given in Theorem \ref{teo2} are Jacobian
varieties of quotients of the Fermat curve $F_p$ by explicit
subgroups $H$ of $G = \textup{Aut}(F_p)$, and computing that
$$
\langle \rho_H , \rho_{\textup{rat}} \rangle = 1
$$
in all cases, where $\rho_H$ is the rational representation of $G$ induced by the trivial representation of $H$ and $\rho_{\textup{rat}}$ is the rational representation of $G$ in $JF_p$, and that
$$
\langle \rho_G , \rho_{\textup{rat}}  \rangle = 0.
$$

It then follows that each factor in the decomposition given in
\eqref{dec1} is isogenous to the corresponding factor in the
decomposition given in Theorem \ref{teo2} as claimed.

\s

b) Observe that the decomposition given  in Theorem \ref{teo1}
is obtained using only the subgroup of order $p^2$ of
Aut($F_p)$. Using the same algorithm mentioned in a) for this
subgroup, we obtain the factors in this decomposition as images
of $JF_p$ under certain concrete primitive rational idempotents
in the rational group algebra of this subgroup. But these
coincide with the endomorphisms considered in \cite{A} to
define the factors in that decomposition.
\end{proof}

\subsection{Some examples}
\subsubsection{}
If $p=7$, then we have only two orbits; $\{1,3,5\}$ and $\{2,4\}$. So in this case
$$JF_{7} \sim JC_{1}^{3} \times JC_{2}^{2} \sim JC_{1}^{3} \times JE_{2}^{6}$$
where $$C_{1}: y^{7}=x(x-1), \quad C_{2}: y^{7}=x^{2}(x-1),\quad E_{2}=C_{2}/\langle R(x,y)=(1/(1-x),-x/y^{3})\rangle\; \mbox{(of genus $1$)}.$$

It is known that $E_2$ is the curve $X_0(49)$.

\s

\subsubsection{}
If $p=11$, then we have only two orbits; $\{1,5,9\}$ and $\{2,3,4,6,7,8\}$. So in this case
$$JF_{11} \sim JC_{1}^{3} \times JC_{2}^{6}$$
where $$C_{1}: y^{11}=x(x-1), \quad C_{2}: y^{11}=x^{2}(x-1).$$


\end{document}